\documentclass[11pt,twoside]{article}

\usepackage{amsmath,amsfonts,bm}

\def\1{\bm{1}}

\DeclareMathAlphabet{\mathsfit}{\encodingdefault}{\sfdefault}{m}{sl}
\SetMathAlphabet{\mathsfit}{bold}{\encodingdefault}{\sfdefault}{bx}{n}

\usepackage[utf8]{inputenc} 
\usepackage[T1]{fontenc}    
\usepackage{booktabs}       
\usepackage{amsfonts}       
\usepackage{nicefrac}       
\usepackage{microtype}      

\usepackage{subfigure}
\usepackage{epsf}
\usepackage{epsfig}
\usepackage{fancyhdr}
\usepackage{graphics}
\usepackage{graphicx}
\usepackage{psfrag}
\usepackage{fullpage}
\usepackage{pdfpages}

\usepackage{url}
\usepackage[colorlinks,linkcolor=magenta,citecolor=blue, pagebackref=true]{hyperref}
\renewcommand*{\backrefalt}[4]{%
    \ifcase #1 \footnotesize{(Not cited.)}%
    \or        \footnotesize{(Cited on page~#2.)}%
    \else      \footnotesize{(Cited on pages~#2.)}%
    \fi}
\usepackage{color}

\usepackage{amsthm}
\usepackage{amsmath}
\usepackage{amssymb,bbm}
\usepackage{caption}
\usepackage{algorithmic}
\usepackage{algorithm}
\usepackage{textcomp}
\usepackage{siunitx}
\usepackage{wrapfig}
\usepackage{algorithmic}
\usepackage{algorithm}
\usepackage{mathrsfs}  
\usepackage{multirow}
\usepackage{multicol}

\newcommand{\kl}{\textnormal{KL}}

\newtheorem{assumption}{Assumption}

\newtheorem{lemma}{Lemma}

\newtheorem{theorem}{Theorem}

\newtheorem{definition}{Definition}

\def\RR{\mathbb{R}}

\def\NN{\mathbb{N}}

\def\FF{\mathbb{F}}

\def\XX{\mathbb{X}}

\def\FF{\mathbb{F}}

\begin{document}

\begin{center}

{\bf{\LARGE{On A Necessary Condition For Posterior Inconsistency:\\ New Insights From A Classic Counterexample}}}
  
\vspace*{.2in}
{\large{
\begin{tabular}{ccc}
Nicola Bariletto & Stephen G. Walker
\end{tabular}
}}

\vspace*{.2in}

\begin{tabular}{c}
The University of Texas at Austin
\end{tabular}

\today


\begin{abstract}
The consistency of posterior distributions in density estimation is at the core of Bayesian statistical theory. Classical work established sufficient conditions, typically combining KL support with complexity bounds on sieves of high prior mass, to guarantee consistency with respect to the Hellinger distance. Yet no systematic theory explains a widely held belief: under KL support, Hellinger consistency is exceptionally hard to violate. This suggests that existing sufficient conditions, while useful in practice, may overlook some key aspects of posterior behavior. We address this gap by directly investigating what must fail for inconsistency to arise, aiming to identify a substantive necessary condition for Hellinger inconsistency. Our starting point is Andrew Barron’s classical counterexample, the only known violation of Hellinger consistency under KL support, which relies on a contrived family of oscillatory densities and a prior with atoms. We show that, within a broad class of models including Barron's, inconsistency requires persistent posterior concentration on densities with exponentially high likelihood ratios. In turn, such behavior demands a prior encoding implausibly precise knowledge of the true, yet unknown data-generating distribution, making inconsistency essentially unattainable in any realistic inference problem. Our results confirm the long-standing intuition that posterior inconsistency in density estimation is not a natural phenomenon, but rather an artifact of pathological prior constructions.
\end{abstract}

\end{center}

\section{Introduction}

This article revisits the foundational issue of posterior consistency in Bayesian density estimation. This has long been a central topic in the theory of asymptotic statistics: the earliest perspective, dating back to Joseph Doob’s foundational work on martingales \cite{doob1949application}, established that the posterior concentrates around the data-generating distribution provided the latter is itself drawn from the prior. A frequentist “what if” validation approach (which we take up here) later became prominent, studying the asymptotic behavior of the posterior when the data are iid from a fixed, unknown distribution. In this setting, Lorraine Schwartz’s breakthrough \cite{schwartz1964consistency} showed that the posterior concentrates in arbitrary weak neighborhoods of the true distribution as long as the latter lies in the Kullback–Leibler (KL) support of the prior (see Definition~\ref{def:KL_support}). Later work showed that inconsistency may arise when this support condition fails \cite{diaconis1986consistency,diaconis1986inconsistent} and extended Schwartz’s result to stronger topologies, for instance as metrized by the Hellinger distance, which are more suitable for density estimation \cite{barron1988exponential, ghosal1999posterior, walker2004newapproaches}. This line of research focused on sufficient conditions, usually by combining KL support with prior mass bounds on high-complexity regions of the parameter space, often in the spirit of nonparametric (sieve) maximum likelihood techniques, to constructively ensure Hellinger consistency.

In this work, we reverse the thought process and ask instead: what must go wrong for the posterior distribution to be inconsistent in the Hellinger sense, assuming KL support (and thus weak consistency) as a minimal well-specification requirement? We view this shift as a way to more effectively identify the substantive obstacles to Hellinger consistency. Once these obstacles are pinned down, it is in turn natural to rule them out directly rather than to indirectly search for conditions implying consistency. As a key step toward this goal, we examine Andrew Barron’s counterexample \cite{barron1988exponential, barron1999consistency} to Hellinger consistency---the only one known, to this day, under the KL support assumption. Our aim is to characterize precisely the mechanism that drives inconsistency in Barron's example and within a broader class of models; once that mechanism is identified, we will proceed to uncover its necessary reliance on an artificial alignment between the prior, the likelihood model, and the data-generating distribution. As a result, we will arrive at the conclusion that, notwithstanding the formal validity of the counterexample, the observed failure should not be a source of practical concern, as the alignment it requires boils down to the modeler possessing knowledge of the unknown data-generating density and using it to construct a tailored, pathological prior.

Taking a step back, we believe it to be a revealing fact that, despite an extensive literature and sophisticated technical tools, only one counterexample is known to produce Hellinger inconsistency under KL support. This rarity itself signals the pathological nature of Hellinger inconsistency, yet within the current state of the literature, this remains only an intuition. The most likely reason for the absence of formal work on this aspect lies in the ingenious yet highly intricate construction of the counterexample. Barron’s model involves a countable family of discontinuous, oscillatory densities that converge weakly, but not in Hellinger distance, to the uniform distribution on $[0,1)$. By assigning carefully chosen prior mass to these densities, the posterior concentrates away from the truth in Hellinger distance while still placing mass near it in the weak topology, in agreement with Schwartz’s theorem. The construction is clearly contrived and leaves, in a non-transparent way, a long list of potential culprits for inconsistency: the prior contains atoms, some densities in its support are discontinuous with infinite KL divergence from the uniform, and the family exhibits a special oscillatory step-function structure.

Therefore, a natural question arises: which of these features is actually responsible for Hellinger inconsistency? Answering this would not only clarify key theoretical aspects of posterior consistency and its failure, but also guide the exclusion of problematic features in model design. Our contribution in this direction is twofold. First, we show that the mechanism behind inconsistency lies in persistent posterior concentration on densities that achieve exponentially high likelihood ratio values in a certain arbitrarily small range. In particular, we prove in Theorem~\ref{thm:beta_boundedness} that this phenomenon is a necessary consequence of Hellinger inconsistency for a broad class of models whose normalized likelihood ratio is bounded a posteriori (see Definition~\ref{def:beta_bounded}), of which Barron’s example is an extreme instantiation. Second, in our main Theorem~\ref{thm:alpha_beta_inconsistency} we demonstrate that this behavior implies a particular alignment between the prior, the likelihood model, and the true distribution. Our analysis shows that such an alignment is highly artificial and essentially unattainable in any model of genuine statistical relevance, as it requires encoding precise knowledge of the unknown data-generating process into a contrived prior specification.

In summary, while the existence of Barron’s ingenious counterexample may appear to justify worries about posterior inconsistency, we demonstrate that the observed failure hinges on an unnatural link between the prior, the model, and the true distribution. By formalizing this failure through new notions of inconsistency (see Definitions~\ref{def:gamma_inconsistency} and \ref{def:alpha_beta_inconsistency}), we demonstrate that it may arise only via design choices that are absent from---indeed infeasible in---standard modeling practice. Overall, the main conceptual contribution of our analysis is to bring fresh evidence that, although Bayesian inconsistency in density estimation may occur in theory through Barron-type constructions, it should not be a practical concern given that it requires a necessary condition that cannot be met in realistic modeling scenarios.

The rest of the article is organized as follows. Section~\ref{sec:notation} introduces the notation and framework used throughout the paper. In Section~\ref{sec:barron}, we describe Barron’s counterexample and uncover a specific form of inconsistency it satisfies. In Section~\ref{sec:necessary_condition}, we relate this to Hellinger inconsistency by proposing a necessary condition which is impossible to meet without knowledge of the true distribution and strong constraints on the prior. Section~\ref{sec:conclusion} concludes the article.

\section{Formal setup and preliminary results}\label{sec:notation}

Consider a family of probability densities $\FF$ with respect to a $\sigma$-finite measure $\lambda$ on a separable metric space $\XX$ with Borel $\sigma$-algebra $\mathscr X$, and equip $\FF$ with the Borel $\sigma$-algebra $\mathscr F$ generated by the topology of weak convergence of probability measures. Let $X_1, X_2, \dots$ be iid random elements from some probability measure $F_\star$ on $(\XX, \mathscr X)$ with density $f_\star$,\footnote{Unless stated otherwise, all probability statements (e.g., “almost surely,” “with positive probability,” etc.) are understood to be under the infinite product measure $F_\star^\infty$ corresponding to the data-generating distribution $F_\star$.} and let $\Pi$ be a prior probability measure on $(\FF, \mathscr F)$. In this setting, Bayes' theorem gives rise to the sequence of posterior distributions
\begin{equation}\label{eq:def_posterior}
    \Pi(\mathrm df\mid  X_{1:n}) = \frac{R_n(f)\, \Pi(\mathrm df)}{\int_\FF R_n(g)\,\Pi(\mathrm dg)}, \quad n\in\NN,
\end{equation}
where $R_n(f):=\prod_{i=1}^n f(X_i)/f_\star(X_i)$ denotes the likelihood ratio at $f$. This general framework accommodates both parametric and nonparametric models (depending on the dimension of $\FF$) on broad classes of sample spaces.

Let $\kl(f, g):= \int_\XX f\ln(f/g)\,\mathrm d\lambda$ for any $f, g \in \FF$. A key condition for posterior consistency, first introduced by \cite{schwartz1964consistency}, is the following.

\begin{definition}\label{def:KL_support}
    A density $f_\star \in \FF$ belongs to the KL support of the prior $\Pi$ if $\Pi(\left\{f \in \FF : \kl(f_\star, f) < \delta\right\}) > 0$  for all $\delta > 0$.
\end{definition}

The simple assumption that $f_\star$ lies in the KL support of $\Pi$ has significant implications for the asymptotic behavior of the posterior distribution as $n \to \infty$. In particular, it implies that the posterior is \emph{weakly consistent}, that is, that it concentrates in any weak neighborhood of $F_\star$ with probability one. When the sample space $\XX$ is a subset of $\RR^d$, this translates to concentration of the posterior around densities with cumulative distribution functions that are close to the true one. Furthermore, the KL support assumption yields a useful lower bound on the denominator of \eqref{eq:def_posterior}, as shown for instance in \cite[lemma 4]{barron1999consistency}. We summarize both results in the lemma below (without proof).

\begin{lemma}\label{lem:schwartz}
    Assume that $f_\star$ lies in the KL support of $\Pi$. Then, for any weak neighborhood $U$ of $F_\star$, $\lim_{n\to\infty} \Pi(U^c \mid X_{1:n}) = 0$ almost surely. Moreover, $\int_\FF R_n(f)\,\Pi(\mathrm df) \geq e^{-\tau n}$ ultimately almost surely, for all $\tau > 0$.
\end{lemma}

If the goal of inference is to estimate the density $f_\star$ rather than the distribution $F_\star$, then weak neighborhoods are not appropriate for assessing consistency. This is because weak convergence of distributions does not, in general, imply convergence of the corresponding densities in any meaningful sense, as distributions with increasingly oscillatory densities may still converge weakly to $F_\star$. To address this, a commonly used stronger topology is that induced by the Hellinger metric $d_h$:  for any $f, g \in \FF$, $d_h^2(f,g) := \int_\XX (f^{1/2} - g^{1/2})^2\, \mathrm d\lambda$. \cite{barron1999consistency, ghosal1999posterior, walker2004newapproaches} were among the first to provide sufficient conditions for \emph{Hellinger consistency}, that is,
\begin{equation*}
    \lim_{n\to\infty} \Pi\left(\left\{f \in \FF : d_h(f, f_\star) > \varepsilon \right\} \mid X_{1:n} \right) = 0
\end{equation*}
almost surely for all $\varepsilon > 0$, by placing suitable constraints on the prior mass assigned to high-complexity regions (sieves) of $\FF$.

\section{Barron's counterexample}\label{sec:barron}

In addition to providing a proof of Hellinger posterior consistency based on sieve upper-bracketing entropy conditions, \cite{barron1999consistency} also presented a counterexample, originally introduced in a solo technical report by Barron \cite{barron1988exponential}, demonstrating that Hellinger consistency may fail even when the KL support condition of Definition~\ref{def:KL_support} (and thus weak consistency) holds. This example effectively shows that KL support is technically not sufficient to guarantee Hellinger consistency. We now describe the counterexample and offer a new perspective on the mechanism driving inconsistency. As mentioned in the introduction, because this is the only known counterexample of Hellinger inconsistency under KL support, understanding what causes inconsistency here, and its consequences, is a crucial step toward analyzing the roots of inconsistency in broader classes of models, which we address in the next section.

Let $\XX = [0,1)$, endowed with the Euclidean topology and its Borel $\sigma$-algebra, and let $\lambda$ denote the Lebesgue measure. Define $\FF := \FF_0 \cup \FF_\triangle$ and let $\Phi$ denote the standard normal cumulative distribution function, where
\begin{equation*}
    \FF_0 := \{f_\theta : \theta \in [0,1]\}, \quad f_\theta(x) := \exp\left\{-\theta + \sqrt{2\theta}\,\Phi^{-1}(x)\right\},
\end{equation*}
and where $\FF_\triangle := \bigcup_{N \in \NN} \FF_N$ is constructed as follows. Let
\begin{equation*}
    P_N := \left\{ \left[\frac{j}{2N^2},\, \frac{j+1}{2N^2} \right) : j = 0,1,\dots,2N^2 - 1 \right\}
\end{equation*}
denote the partition of $\XX$ into $2N^2$ consecutive intervals of equal length. Then, $\FF_N$ is the collection of all $\binom{2N^2}{N^2}$ densities that take constant value $2$ on exactly $N^2$ of the intervals in $P_N$, and value $0$ on the remaining $N^2$ intervals.

Assume that the true density is $f_\star = f_0$, the uniform density on $[0,1)$. The prior $\Pi$ is chosen to place half of its mass on $\FF_0$ according to
%
 $   \Pi_0(\mathrm d\theta) \propto e^{-1/\theta}\,1_{[0,1]}(\theta)\,\mathrm d\theta,$
%
and the remaining half on $\FF_\triangle$, via
\begin{equation*}
    \Pi_\triangle(\mathrm df) \propto \sum_{N \in \NN} \frac{1}{N^2} \sum_{g \in \FF_N} \frac{1}{\binom{2N^2}{N^2}}\,\delta_g(\mathrm df).
\end{equation*}
That is, the prior splits its mass evenly: one half is assigned diffusely to the set of continuous densities $\FF_0$, while the other half is assigned to the oscillatory densities in $\FF_\triangle$ by (i) allocating mass proportional to $N^{-2}$ to each set $\FF_N$, and (ii) distributing that mass uniformly across the members of $\FF_N$.

\subsection{Inconsistency}

The key result of \cite{barron1988exponential} is that, although the prior diffusely assigns mass within $\FF_0$ in such a way that the KL support condition is satisfied at $f_\star$ (thus guaranteeing weak consistency), the posterior distribution is Hellinger inconsistent. Specifically, the posterior mass assigned to $\FF_\triangle$, whose elements remain at a constant positive Hellinger distance from the uniform density, accumulates arbitrarily close to 1 infinitely often, almost surely. This is possible because the densities in $\FF_\triangle$ are able to approximate the uniform distribution weakly while remaining far in Hellinger distance from the uniform density. When paired with a carefully constructed prior that places enough mass on fast-oscillating densities within $\FF_\triangle$—which can closely track the weakly converging empirical distribution—the result is Hellinger inconsistency even in the presence of weak consistency.

As it is clear from the construction, the model and prior are cleverly engineered to induce pathological behavior and ultimately inconsistency: the members of $\FF_\triangle$ feature severe and discontinuous oscillations, receiving high prior mass, in a way that weakly targets the uniform distribution through densities that are far from the uniform density. The pathological nature of the model, and in particular its deliberate targeting of the uniform distribution as a candidate data-generating process, is made even more apparent by the following lemma.

\begin{lemma}
    Barron's model is Hellinger consistent at all $f\in\mathbb F\setminus\{f_0\}$.
\end{lemma}
\begin{proof}
    For any $\theta_1, \theta_2 \in [0,1]$, one easily checks that $\kl(f_{\theta_1}, f_{\theta_2}) = (\theta_2 - \theta_1) + \sqrt{2\theta_1}(\sqrt{2\theta_1} - \sqrt{2\theta_2})$, which vanishes as $\lvert \theta_1 - \theta_2 \rvert \to 0$. Hence, every KL neighborhood of $f = f_{\theta_\star}$, for $\theta_\star \in [0,1]$, contains a Euclidean neighborhood of the form $\{f_\theta \in \FF_0 : \lvert \theta - \theta_\star \rvert < \delta\}$ for some $\delta>0$. This observation, combined with the full support of $\Pi_0$, implies that any $f \in \FF_0$ belongs to the KL support of the prior. The same holds trivially for any $f \in \FF_\triangle$, since each of these densities is an atom of $\Pi_\triangle$. Therefore, by Schwartz's theorem, the posterior is weakly consistent at all $f \in \FF$.

    Now take $f \neq f_0$. If $f \in \FF_\triangle$, then $\Pi$ assigns positive mass to $\{f\}$, so by Doob's consistency theorem, the posterior is Hellinger consistent at $f$. If instead $f = f_\theta$ for some $\theta \in (0,1]$, note that no sequence $(f^k)_{k \in \NN} \in \FF^\NN$ can weakly converge to $f$ unless it also converges to $f$ in Hellinger distance. Indeed, $\FF_0$ is identifiable and the Hellinger distance between its elements depends continuously on the Euclidean distance between their parameters, so no sequence eventually contained in $\FF_0$ may converge weakly to $f_\theta$ without converging to it in the Hellinger sense; moreover, for every $\theta \in (0,1]$, there exists $x_\theta \in (0,1)$ such that $f_\theta(x) > 2 = \max_{g \in \FF_\triangle,\, x \in [0,1)} g(x)$ for all $x \ge x_\theta$. Hence, no sequence eventually contained in $\FF_\triangle$ can weakly approximate $f_\theta$. These facts, together with weak consistency, yield Hellinger consistency at $f$.
\end{proof}

The above lemma shows that, despite the pathological construction behind Barron’s model, consistency is challenged only at the uniform distribution, through oscillatory densities that target it specifically. This provides an initial glimpse into the paradoxical nature of inconsistency: the latter requires the statistician to construct a Bayesian model that reflects precise knowledge of the true distribution, which is, by the very nature of statistical inference, unknown to them. To sharpen our understanding of this paradox, we now introduce a new notion of inconsistency—one that the counterexample is later shown to satisfy—and examine the structural implications of this behavior.

\begin{definition}\label{def:gamma_inconsistency}
    Let $\gamma>0$. The posterior distribution is $\gamma$-\emph{inconsistent} if
    \[
        \limsup_{n\to\infty} \Pi\left(\left\{f \in \FF : R_n(f) = e^{\gamma n} \right\} \mid X_{1:n} \right) = 1.
    \]
    almost surely. Equivalently, $\Pi\left(\left\{f \in \FF : R_n(f) = e^{\gamma n} \right\} \mid X_{1:n} \right) > \delta$ infinitely often almost surely, for any $\delta<1$.
\end{definition}

In words, $\gamma$-inconsistency requires the posterior distribution to assign persistently high mass to densities that yield a value of $R_n^{1/n}(f)$ exactly equal to $e^\gamma>1$. Intuitively, this phenomenon may be interpreted as a form of inconsistency because the densities in question are far, in terms of likelihood ratio, from $f_\star$. With the next result, we show that this behavior imposes an unnaturally strong connection between the likelihood model, the true distribution, and the prior.

\begin{theorem}\label{thm:gamma_inconsistency}
    Assume that $f_\star$ lies in the KL support of $\Pi$, and that the posterior distribution is $\gamma$-inconsistent. Then, almost surely, $\gamma$ is an accumulation point of the sequence
    \[
        -n^{-1} \ln \Pi\left(\left\{f \in \FF : n^{-1}\ln R_n(f) = \gamma \right\}\right), \quad n \in \NN.
    \]
\end{theorem}

\begin{proof}
    A very similar line of proof as in Theorem~\ref{thm:alpha_beta_inconsistency} below yields
    \begin{equation*}
        \gamma - \tau_1 \leq -n^{-1}\ln\Pi(\{f\in\mathbb F : R_n(f) = e^{\gamma n}\}) \leq \gamma + \tau_2
    \end{equation*}
    infinitely often almost surely, for all $\tau_1, \tau_2>0$. Therefore,
    \begin{equation*}
        F_\star^\infty\left(\bigcap_{k\in\NN}\Big\{\gamma - k^{-1} \leq -n^{-1}\ln\Pi(\{f\in\mathbb F : R_n(f) = e^{\gamma n}\}) \leq \gamma + k^{-1} \quad \textnormal{i.o.}\Big\}\right) = 1.
    \end{equation*}
    This implies the conclusion of the theorem.
\end{proof}

In other words, $\gamma$-inconsistency implies that the prior mass assigned to densities that achieve a specific likelihood ratio value dictated by $\gamma$, must have an asymptotic behavior directly related to $\gamma$ itself. In practice, this is an unachievable scenario because (i) it requires a close alignment between the prior and the data, (ii) it necessitates a strong connection between the prior and the true, in practice unknown distribution $F_\star^\infty$ (as it appears both in the likelihood ratio and in the ``almost surely'' qualifier), and (iii) commonly employed full-support priors cannot satisfy the exact relation described in the theorem (unless they are carefully engineered via strategically placed atoms at oscillatory densities), since sets of the form $\{f \in \FF : n^{-1} \ln R_n(f) = \gamma\}$ define proper, data-dependent subspaces of $\FF$.

We are now ready to state our main result concerning the counterexample, leading to the conclusion that the inconsistency observed in Barron's model is tied to an artificial design.

\begin{theorem}
    Barron's counterexample is $\ln 2$-inconsistent at the uniform density.
\end{theorem}

\begin{proof}
    Let $B_n:= \{f : R_n(f) = 2^n\}$. We will show that $\limsup_{n\to\infty}\Pi(B_n\mid X_{1:n})=1$ almost surely. Equation 30 in \cite{barron1999consistency} shows that
    \begin{equation*}
        \int_{B_n} R_n(f)\, \Pi(\mathrm df) \geq \int_{B_n\cap\mathbb F_\triangle} R_n(f)\, \Pi(\mathrm df) \geq a_1 n^{-1}
    \end{equation*}
    ultimately almost surely for some universal constant $a_1>0$. Moreover
    \begin{align*}
        \int_{B_n^c} R_n(f)\, \Pi(\mathrm df)  \leq \int_{\mathbb F_0} R_n(f)\, \Pi(\mathrm df) + \int_{B_n^c\cap \mathbb F_\triangle} R_n(f)\, \Pi(\mathrm df),
    \end{align*}
    Equation 31 in \cite{barron1999consistency} shows that the first term is smaller than $a_2\exp\{-2\sqrt n\}$ infinitely often almost surely, for $a_2>0$ another universal constant. As for the second term, notice that, for $f\in\mathbb F_\triangle\cap B_n^c$, we have $R_n(f)=0$: this is because $R_n(f)$ is at most $2^n$ (when $f(X_i)=2$ for all $i=1,\dots,n$), and $R_n(f)=0$ otherwise. Therefore
    \begin{equation*}
        \frac{\Pi(B_n \mid X_{1:n})}{\Pi(B_n^c \mid X_{1:n})} = \frac{\int_{B_n} R_n(f)\, \Pi(\mathrm df)}{\int_{B_n^c} R_n(f)\, \Pi(\mathrm df)} \geq \frac{a_1\exp\{2\sqrt n\}}{a_2 n}
    \end{equation*}
    infinitely often almost surely. This implies $\limsup_{n\to\infty}\Pi(B_n\mid X_{1:n})=1$ almost surely, or equivalently $\Pi(B_n\mid X_{1:n})>\delta$ infinitely often almost surely for all $\delta\in(0,1)$.
\end{proof}

\section{A necessary condition for Hellinger inconsistency}\label{sec:necessary_condition}

Intuitively, it is clear that $\gamma$-inconsistency is intimately connected with Hellinger inconsistency in Barron's counterexample: because a large prior mass is placed on densities that achieve a constant and high value of $R_n^{1/n}(f)$, the posterior concentrates significant mass on them persistently, resulting in Hellinger inconsistency. To formalize this intuition, we slightly weaken the notion of $\gamma$-inconsistency and proceed to show that, for a certain class of models encompassing Barron's, this new form of inconsistency is a necessary consequence of Hellinger inconsistency.

\begin{definition}\label{def:alpha_beta_inconsistency}
    Let $0<\alpha\leq \beta<\infty$. The posterior distribution is \emph{$(\alpha,\beta)$-inconsistent} if there exists $\delta>0$ such that
    \begin{equation*}
        F_\star^\infty\bigg(\Pi(\{f\in\FF : e^{\alpha n}\leq R_n(f)\leq e^{\beta n}\}\mid X_{1:n})>\delta \quad \textnormal{i.o.}\bigg) > 0.
    \end{equation*}
\end{definition}
In words, an $(\alpha,\beta)$-inconsistent posterior has a positive probability to place persistently high mass on densities that generate exponentially high likelihood ratio values lying in a certain fixed range. Clearly, $\gamma$-inconsistency, as satisfied by Barron's model, is more extreme than $(\alpha, \beta)$-inconsistency, for three reasons: first, it requires the posterior to put mass on densities that yield a normalized likelihood ratio value exactly equal to a certain value $e^\gamma$, rather than lying in a range $[e^\alpha,e^\beta]$; second, it requires the posterior mass of this pathological set of densities to be larger than any $\delta<1$ infinitely often, rather than only for some possibly small $\delta>0$; finally, it requires this to happen with probability 1, rather than just with positive probability. Notice that the shift to $(\alpha,\beta)$-inconsistency also naturally broadens the scope of our analysis to diffuse priors, since it is difficult to envision how any diffuse prior could satisfy the exact equalities required by the stricter notion of $\gamma$-inconsistency.

Before proceeding with the analysis, however, it is important to emphasize that, although $(\alpha,\beta)$-inconsistency was argued to be milder than $\gamma$-inconsistency, its strength and proximity to the latter should not be underestimated. Indeed, a simple argument—based on iteratively bisecting the interval $[e^\alpha,e^\beta]$, see the proof of Theorem~\ref{thm:alpha_beta_inconsistency} below—shows that the existence of such a range essentially implies the existence of a point $\gamma \in [\alpha,\beta]$ such that the posterior places persistently high mass on densities with a normalized likelihood ratio converging (along a subsequence) to $e^\gamma$. This is weaker than $\gamma$-inconsistency, in that it allows coincidence with $e^\gamma$ only approximately rather than exactly, but it is closely related to it in spirit.

The next step is to define a fairly general class of models for which, as we will see, $(\alpha,\beta)$-inconsistency is particularly useful to study the consequences of Hellinger inconsistency.

\begin{definition}\label{def:beta_bounded}
    Let $\beta\geq 0$. A Bayesian model $(\FF, \Pi)$ is \emph{$\beta$-bounded under $F_\star$} if $\lim_{n\to\infty}\Pi(\{f\in\mathbb F : R_n(f)> e^{\beta n}\} \mid X_{1:n}) = 0$ almost surely under $F_\star^\infty$.
\end{definition}

In essence, $\beta$-boundedness implies posterior concentration on bounded values of the normalized likelihood ratio $R_n^{1/n}(f)$. This is the case, for instance, in models with bounded $f(x)/f_\star(x)$. It is not hard to see that Barron's counterexample is in fact $\beta$-bounded under the uniform distribution, with $\beta=\ln 2$, since the oscillating densities in $\mathbb F_\triangle$ achieve a likelihood of at most 2. More formally, let $B_n:=\{f\in\FF : R_n(f)>2^n\, \textnormal{ i.o.}\}$, so $\Pi(B_n\mid X_{1:n}) = \Pi(B_n\cap \FF_0\mid X_{1:n})$ and, by simple maximum likelihood calculations, we have
\begin{align*}
    \left\{\sup_{\theta\in[0,1] }\prod_{i=1}^n f_\theta(X_i)>e^{\beta n}\quad \textnormal{i.o.}\right\} &= \left\{\exp\left\{\frac{n}{2}W_n^2\right\}>e^{\beta n},\, W_n>0\quad \textnormal{i.o.}\right\} \\
    & \subseteq  \left\{W_n^2>2\beta\quad\textnormal{i.o.}\right\},
\end{align*}
where $W_n:=n^{-1}\sum_{i=1}^n \Phi^{-1}(X_i)\to 0$ almost surely by the strong law of large numbers. This shows that $B_n\cap \mathbb F_0$ is ultimately empty when restricted to a set of probability one, implying that Barron's model is $\beta$-bounded under $F_\star$ for $\beta=\ln2$.

We are now in a position to state an important result that reveals a strong link between Hellinger and $(\alpha,\beta)$-inconsistency for $\beta$-bounded models.

\begin{theorem}\label{thm:beta_boundedness}
    Assume that the model $(\FF, \Pi)$ is $\beta$-bounded under $F_\star$ for some $\beta>0$. Then, if $f_\star$ is in the KL support of $\Pi$, the posterior is Hellinger inconsistent only if it is $(\alpha,\beta)$-inconsistent for some $\alpha\in(0, \beta]$.
\end{theorem}

\begin{proof}
    We proceed by contraposition. Assume that the posterior is not $(\alpha,\beta)$-inconsistent for any $\alpha\in(0,\beta]$, so $\Pi(\{f\in \FF : R_n(f) \geq  e^{\alpha n}\}\mid X_{1:n})\to 0$ almost surely for all $\alpha>0$. Then, denoting $A_\varepsilon:= \{f\in\FF : d_h(f, f_\star)<\varepsilon\}$ for any $\varepsilon>0$, we have
    \begin{equation*}
        \Pi(A_\varepsilon^c\mid X_{1:n}) \leq \Pi(\{f\in\FF : R_n(f) \geq e^{\alpha n}\}) \,+\, \Pi(A_\varepsilon^c \cap \{f\in\FF : R_n(f) < e^{\alpha n}\}\mid X_{1:n}),
    \end{equation*}
    where the first addendum converges to 0 almost surely for all $\alpha>0$. As for the second one, Lemma \ref{lem:schwartz} implies
    \begin{equation*}
        \Pi(A_\varepsilon^c \cap \{f\in\FF : R_n(f) < e^{\alpha n}\}\mid X_{1:n}) \leq e^{\tau n} e^{\alpha n/2} \int_{A_\varepsilon^c} R_n^{1/2}(f) \,\Pi(\mathrm df)
    \end{equation*}
    ultimately almost surely for all $\tau>0$. By standard calculations, the integral can be shown to be smaller than $e^{-n\varepsilon^2/4}$ ultimately almost surely, so choosing $\alpha$ and $\tau$ small enough ensures the almost sure vanishing of $\Pi(A_\varepsilon^c\mid X_{1:n})$ as $n\to\infty$. As this holds for all $\varepsilon>0$, Hellinger consistency is concluded.
\end{proof}

Therefore, studying the consequences of $(\alpha,\beta)$-inconsistency provides insight into the mechanisms underlying Hellinger inconsistency. The next theorem,\footnote{To keep the notation concise, the statement of Theorem~\ref{thm:alpha_beta_inconsistency} focuses on the case $\alpha < \beta$, though the case $\alpha = \beta$ can be handled analogously to Theorem~\ref{thm:gamma_inconsistency}.} the central one of the article, shows that $(\alpha,\beta)$-inconsistency entails an artificial dependence of the prior on the likelihood model and the in practice unknown $F_\star$: the prior must assign mass to arbitrarily narrow intervals of values of $n^{-1}\ln R_n(f)$, with the required mass itself lying in such intervals. Not surprisingly, this mirrors the phenomenon observed in the more extreme case of $\gamma$-inconsistency.

\begin{theorem}\label{thm:alpha_beta_inconsistency}
    Assume that $f_\star$ is in the KL support of $\Pi$ and that, for some $0<\alpha < \beta<\infty$, the posterior is $(\alpha,\beta)$-inconsistent. Then there exists $\gamma\in[\alpha,\beta]$ such that, for all small $\chi>0$, one can find $\mu_1,\mu_2\geq 0$ with $\mu_1+\mu_2\leq \chi$ and
    \begin{equation}\label{eq:theorem_alpha_beta_inconsistency}
        \gamma - \mu_1 - \tau_1 \leq -n^{-1}\ln\Pi\big(\big\{f\in\FF : \gamma-\mu_1 \leq n^{-1}\ln R_n(f)\leq \gamma + \mu_2\big\}\big) \leq \gamma + \mu_2 + \tau_2
    \end{equation}
    infinitely often with positive probability, for all $\tau_1,\tau_2>0$.
\end{theorem}

\begin{proof}
    Without loss of generality, pick $\chi=1/k$ for $k\in\NN$ large enough. For some $\delta>0$, we have
    \begin{align*}
        \delta & < \Pi(\{f\in\FF : e^{\alpha n}\leq R_n(f)\leq e^{\beta n}\}\mid X_{1:n}) \\
        & \leq   \Pi(\{f\in\FF : e^{\alpha n}\leq R_n(f)\leq e^{cn}\}\mid X_{1:n}) \, + \, \Pi(\{f\in\FF : e^{cn}\leq R_n(f)\leq e^{\beta n}\}\mid X_{1:n})
    \end{align*}
    infinitely often with positive probability, where $c:=(\alpha+\beta)/2$. Therefore, at least one addendum must be greater than some positive constant infinitely often with positive probability. One can iterate this argument $\lceil\log_2((\alpha-\beta)/\chi)\rceil$ times to find an interval of length $\xi_\chi\leq \chi$, say $[\gamma_\chi - \xi_\chi/2,\, \gamma_\chi+\xi_\chi/2]$ for some $\gamma_\chi\in[\alpha,\beta]$, such that
    \begin{equation*}
        \Pi(\{f\in\FF : e^{(\gamma_\chi - \xi_\chi/2) n}\leq R_n(f)\leq e^{(\gamma_\chi+\xi_\chi/2) n}\}\mid X_{1:n}) > \delta_\chi
    \end{equation*}
    infinitely often with positive probability, for some $\delta_\chi>0$.

    When $B_{\chi,n}:=\big\{f\in\FF : e^{(\gamma_\chi - \xi_\chi/2) n}\leq R_n(f)\leq e^{(\gamma_\chi+\xi_\chi/2) n}\big\}$, because Lemma~\ref{lem:schwartz} yields $\int_\mathbb F R_n(f)\, \Pi(\mathrm df)\geq e^{-\tau n}$ ultimately almost surely for all $\tau>0$, we obtain
    \begin{equation*}
        \int_{B_{\chi,n}} R_n(f)\, \Pi(\mathrm d f)\geq e^{-\tau n}
    \end{equation*}
    infinitely often with positive probability for all $\tau>0$. Moreover, by a simple application of Markov's inequality and the first Borel-Cantelli lemma, we have
    \begin{equation*}
        \int_{B_{\chi,n}} R_n(f)\, \Pi(\mathrm d f)\leq \int_{\mathbb F} R_n(f)\, \Pi(\mathrm d f)\leq e^{\tau n}
    \end{equation*}
    ultimately almost surely for all $\tau>0$. By definition of $B_{\chi,n}$, we also have that
    \begin{equation*}
        e^{(\gamma_\chi - \xi_\chi/2)n}\, \Pi(B_{\chi,n})\leq \int_{B_{\chi,n}} R_n(f)\, \Pi(\mathrm d f)\leq e^{(\gamma_\chi + \xi_\chi/2)n}\, \Pi(B_{\chi,n}).
    \end{equation*}
    Therefore, combining all of the above yields
    \begin{equation*}
        \gamma_\chi - \xi_\chi/2 - \tau_1 \leq -n^{-1}\ln\Pi(B_{\chi,n}) \leq \gamma_\chi + \xi_\chi/2 + \tau_2
    \end{equation*}
    infinitely often with positive probability, for all $\tau_1, \tau_2>0$. Now  $(\gamma_{1/k})_{k\in\NN}$ is a Cauchy sequence contained in the closed set $[\alpha,\beta]$, so it has a limit $\gamma\in[\alpha,\beta]$.
    Defining $\mu_1:=\gamma - (\gamma_\chi-\xi_\chi/2)$ and $\mu_2:=(\gamma_\chi+\xi_\chi/2)-\gamma$, condition \eqref{eq:theorem_alpha_beta_inconsistency} follows.
\end{proof}






The structure uncovered by Theorem~\ref{thm:alpha_beta_inconsistency} is impossible to achieve in genuine statistical inference problems, primarily because it requires knowledge of the true density, which appears both in $n^{-1}\ln R_n(f)$ (a finite-sample version of the unknown entropy number $-\kl(f_\star,f)$) and in the ``positive probability'' statement. Moreover, even assuming (unrealistically) that $f_\star$ is known, encoding this knowledge into the prior in a way that satisfies condition~\eqref{eq:theorem_alpha_beta_inconsistency} is extremely difficult to conceive for two main reasons:
\begin{enumerate}
    \item Condition~\eqref{eq:theorem_alpha_beta_inconsistency} can be loosely but intuitively interpreted as implying that
    \begin{equation*}\label{unreal}
         -n_k^{-1}\ln\Pi\left(\left\{f\in\FF :  n_k^{-1}\ln R_{n_k}(f) \to \gamma \right\} \right) \to \gamma
    \end{equation*}
    along a subsequence $(n_k)_{k\in\NN}$, since the interval length $\chi>0$ can be taken arbitrarily small. This is close in spirit to the extreme case of $\gamma$-inconsistency, where the inner limit is replaced by a finite-$n$ equality. As such, the implications for that case---e.g., that such sets live in low-dimensional subspaces of $\FF$ and would require positive prior mass essentially achievable only through carefully placed atoms---can be translated to this more general setting. 

    \item Unless the model is constructed exactly as in Barron’s example, where the log-likelihood ratio persistently hits $\gamma$ exactly, the prior constraint in~\eqref{eq:theorem_alpha_beta_inconsistency} must hold \emph{simultaneously} across all $\chi$. This imposes an additional and extreme level of stringency on the model design required to induce Hellinger inconsistency.
\end{enumerate}

In summary, the two observations above and Theorem~\ref{thm:alpha_beta_inconsistency} itself clarify that, unless one knows the true distribution exactly and designs the prior with extreme ingenuity to satisfy a large number of highly contrived constraints, posterior inconsistency in $\beta$-bounded models is automatically ruled out.


\section{Conclusion}\label{sec:conclusion}

We shed light on the implications of posterior inconsistency in density estimation by precisely characterizing this phenomenon in Barron’s counterexample and for a wider class of models. Our results show that Barron’s construction, though formally valid, depends on an artificially strong alignment between the prior, the likelihood, and the data-generating process—an alignment unattainable in practice without precise knowledge of the true distribution encoded into the prior. Thus, while ingenious, the failure mode underlying the counterexample poses no genuine threat to posterior consistency in realistic settings. Moreover, our notion of $\gamma$-inconsistency, satisfied by the counterexample, isolates the mechanism at work: it emerges as an extreme instance of $(\alpha,\beta)$-inconsistency, which we establish as a necessary consequence of Hellinger inconsistency for a broad class of models that includes Barron’s.

Our analysis, primarily aimed at understanding the counterexample and its implications, is restricted to $\beta$-bounded models. Although this excludes cases where the posterior may not concentrate on bounded normalized likelihood ratios, we note that, as Barron’s $\ln 2$-bounded example shows, inconsistency does not fundamentally rely on unbounded ratios. Rather, it arises from (i) carefully engineered oscillatory behavior approximating the true distribution away from the true density, and (ii) unrealistic prior designs requiring knowledge of the truth. Hence, focusing on $\beta$-bounded models is sufficient to further the current understanding of the core mechanisms underlying inconsistency. This point is reinforced by looking at a recent contribution by \cite{bariletto2025identifiability}, who provide sufficient conditions for parametric posterior consistency and illustrate their theory via the one-dimensional family $f_\theta(x)\propto(1+\cos(\theta x))1_{[0,1]}(x),$ for $\theta\geq 0.$ Using a simple identifiability argument, \cite{bariletto2025identifiability} show that the posterior is consistent at any $\theta>0$ under any full-support prior. Instead, when the data are generated from the uniform distribution (corresponding to both $\theta=0$ and $\theta\to\infty$, due to cosine oscillations), inconsistency may arise. Note that, when $f_\star$ is uniform, the likelihood ratio satisfies $f_\theta(x)/f_\star(x)\leq e^\beta$ uniformly over $\theta\geq 0$ and $x\in[0,1]$ for some finite $\beta>0$, so the model fits within our $\beta$-bounded framework.

\subsection{Future work}

We conclude by noting that promising future directions arise from comparing Barron’s model with the parametric formulation of \cite{bariletto2025identifiability}. In light of Theorem~\ref{thm:alpha_beta_inconsistency}, while the cosine-based model may be inconsistent at $f_\star=f_0$ due to oscillations at infinity, it remains unclear how to construct a prior with full KL support that yields an inconsistent posterior. Indeed, \cite{bariletto2025identifiability} show that consistency at the uniform density holds under extremely mild tail conditions on the prior for $\theta$. We conjecture that this is because, although oscillatory densities weakly targeting $f_0$ exist in the model’s tail, no reasonable (e.g., continuous and eventually decreasing) prior can average them out to recover the true uniform density. In other words, it seems unlikely that a prior or posterior predictive density derived from this model could equal the true density when averaging occurs away from the true parameter. This contrasts with Barron’s model, where the symmetrically increasing oscillations and their uniformly distributed prior mass yield a prior predictive density restricted to $\FF_\triangle$ that is exactly uniform (as one easily checks), and a posterior predictive that, with positive probability, must converge to the uniform density in Hellinger distance (at least along a subsequence), by Theorem~1 in \cite{walker2004newapproaches}. These observations suggest that focusing on this predictive perspective more closely could further clarify the interactions between density oscillations and the prior, and lead to refined results for models of practical interest including non-$\beta$-bounded ones such as nonparametric mixtures.

\bibliography{arXiv.bib}
\bibliographystyle{apalike}

\end{document}